\documentclass[11pt]{article}
\usepackage{amsmath}
\usepackage{amsthm}
\usepackage{amssymb,amsfonts}
\usepackage[T1]{fontenc}
\usepackage{hyperref}
\usepackage{latexsym}
\usepackage{todonotes}
\usepackage{upgreek}
\usepackage{rotating}
\usepackage{nicefrac}
\usepackage{epsfig}
\usepackage{stmaryrd}
\usepackage{setspace}
\usepackage{enumerate}
\usepackage[all]{xypic}
\usepackage{bbm,ifpdf,tikz}
\ifpdf
\usepackage{pdfsync}
\fi

\newtheorem{theorem}{Theorem}[section]

\newtheorem{lemma}[theorem]{Lemma}
\newtheorem{proposition}[theorem]{Proposition}
\newtheorem{corollary}[theorem]{Corollary}

{
\theoremstyle{definition}

\newtheorem{example}[theorem]{Example}

\newtheorem{remark}[theorem]{Remark}
}

\newcommand{\excise}[1]{}
\newcommand{\Spec}{\operatorname{Spec}}

\newcommand{\Ext}{\operatorname{Ext}}

\newcommand{\id}{\operatorname{id}}

\renewcommand{\dim}{\operatorname{dim}}

\newcommand{\crk}{\operatorname{crk}}

\renewcommand{\and}{\qquad\text{and}\qquad}

\newcommand{\Hom}{\operatorname{Hom}}

\newcommand{\bigmid}{\;\big{|}\;}
\newcommand{\Bigmid}{\;\Big{|}\;}
\newcommand{\Edge}{\operatorname{Edge}}
\renewcommand{\Vert}{\operatorname{Vert}}
\newcommand{\uk}{\underline{k}}
\newcommand{\gr}{\operatorname{gr}}

\newcommand{\Z}{\mathbb{Z}}

\newcommand{\N}{\mathbb{N}}
\newcommand{\R}{\mathbb{R}}
\newcommand{\C}{{\mathcal{C}}}

\newcommand{\IH}{\operatorname{IH}}

\newcommand{\cG}{\mathcal{G}}
\newcommand{\cGgred}{\cG_{g,\operatorname{red}}}
\newcommand{\cGg}{\cG_{g}}

\newcommand{\cC}{\mathcal{C}}

\newcommand{\cRT}{\mathcal{RT}}
\newcommand{\cPT}{\mathcal{PT}}

\newcommand{\cPGgop}{\mathcal{PG}_g^\op}
\newcommand{\cRGgop}{\mathcal{RG}_g^\op}
\newcommand{\cPGg}{\mathcal{PG}_{\!g}}

\newcommand{\cO}{\mathcal{O}}
\newcommand{\cOg}{\cO_g}
\newcommand{\cOgsm}{\cO_g^{\operatorname{small}}}
\newcommand{\cOgti}{\cO_g^{\operatorname{tiny}}}

\newcommand{\OS}{\operatorname{OS}}

\newcommand{\Rep}{\operatorname{Rep}}

\newcommand{\Aut}{\operatorname{Aut}}

\newcommand{\cGgop}{\cGg^{\op}}
\newcommand{\cGopg}{\cGg^{\op}}

\newcommand{\cGop}{\mathcal{G}^{\op}}

\newcommand{\op}{{\operatorname{op}}}

\renewcommand{\Vert}{\operatorname{Vert}}
\newcommand{\Out}{\operatorname{Out}}
\newcommand{\Mor}{\operatorname{Mor}}
\newcommand{\ue}{\underline{e}}
\newcommand{\uf}{\underline{f}}
\newcommand{\um}{\underline{m}}
\newcommand{\un}{\underline{n}}
\newcommand{\uv}{\underline{v}}
\newcommand{\OIr}{\operatorname{OI}^r}

\newcommand{\UConf}{\operatorname{UConf}}
\newcommand{\Sw}{\widetilde{S}}
\newcommand{\tS}{\widetilde{S}}

\newcommand{\ith}{i^{\operatorname{th}}}

\newcommand{\banana}{
      \begin{sideways}$\hspace{-1pt}\ominus$\end{sideways}}

\DeclareMathOperator{\FI}{FI}

\DeclareMathOperator{\OI}{OI}

\newcommand{\PTi}{\mathbf{PT}}
\newcommand{\RTi}{\mathbf{RT}}

\begin{document}
\spacing{1.2}
\noindent{\LARGE\bf The contraction category of graphs}\\

\noindent{\bf Nicholas Proudfoot and Eric Ramos}\\
Department of Mathematics, University of Oregon,
Eugene, OR 97403\\

{\small
\begin{quote}
\noindent {\em Abstract.}
We study the category whose objects are graphs of fixed genus and whose morphisms
are contractions.  We show that the corresponding contravariant module categories are Noetherian
and we study two families of modules over these categories.  
The first takes a graph to a graded piece of the homology
of its unordered configuration space and the second takes a graph to an intersection homology group whose dimension
is given by a Kazhdan--Lusztig coefficient; in both cases we prove that the module is finitely generated.
This allows us to draw conclusions about 
torsion in the homology groups of graph configuration spaces, and about the growth of Betti numbers
of graph configuration spaces and Kazhdan--Lusztig coefficients of graphical matroids.
We also explore the relationship between our category and outer space, which is used in the study of outer
automorphisms of free groups.
\end{quote} }

\section{Introduction}
We are interested in ways of assigning a vector space or abelian group to a graph that are contravariantly functorial with respect to contractions of graphs.
A contraction, which is defined precisely in Section \ref{sec:main defs}, preserves the genus (first Betti number) of a graph,
so we consider the category $\cGg$ whose objects are graphs of genus $g$ and whose morphisms are contractions.
For any commutative ring $k$, we define $\Rep_k(\cGgop)$ to be the category of functors from $\cGgop$ to $k$-modules.
An object of this category is called a {\bf \boldmath{$\cGgop$}-module with coefficients in \boldmath{$k$}}.

\subsection{Noetherianity and growth}\label{intro-growth}
For any category $\C$, a module $M\in\Rep_k(\C)$ is called {\bf finitely generated} if 
there exist finitely many objects $x_1,\ldots,x_r$ of $\C$ along with elements $v_i\in M(x_i)$ such that, for any object
$x$ of $\C$, $M(x)$ is spanned over $k$ by the images of the elements $v_i$ along the maps induced by 
all possible morphisms $f_i:x_i\to x$.  If every submodule of a finitely generated module is itself finitely generated, 
the category $\Rep_k(\C)$ is said to be {\bf locally Noetherian}.

Sam and Snowden have developed powerful machinery for proving that module categories are locally Noetherian.
They define what it means for $\C$ to be {\bf quasi-Gr\"obner}, and they show that, if $\C$ is quasi-Gr\"obner, then
$\Rep_k(\C)$ is locally Noetherian for any Noetherian commutative algebra $k$ \cite{sam}.  The most prominent example
of a quasi-Gr\"obner category is the category $\FI$ of finite sets with injections; the fact that $\Rep_k(\FI)$ is locally Noetherian
has been used to prove stability patterns in coinvariant algebras and in the cohomology groups of configuration spaces 
and other moduli spaces \cite{CEF}, in the homology groups of congruence subgroups \cite{Put}, and in the syzygies of Segre embeddings \cite{Sno}.  

In the prequel to this paper, the authors built on work of Barter \cite{Barter} to prove that the opposite category $\cG_0^\op$
of trees with contractions is quasi-Gr\"obner \cite{PR-trees}.  The technical heart of this paper is the extension of this
result to arbitrary genus.

\begin{theorem}\label{noetherian}
For any non-negative integer $g$,
the category $\cGgop$ is quasi-Gr\"obner, and therefore the category $\Rep_k(\cGgop)$ is locally Noetherian
for any Noetherian commutative algebra $k$.
\end{theorem}

Theorem \ref{noetherian} is useful for proving that specific $\cGgop$-modules are finitely generated, and this gives some control
over their dimension growth.  More precisely, we say that a module is {\bf finitely generated in degrees \boldmath{$\leq d$}}
if the objects $x_1,\ldots,x_r$ in the definition of finite generation may be taken to be graphs with at most $d$ edges.
If $k$ is a field and $M$ is finitely generated in degrees $\leq d$, then the dimension of $M(G)$ is constrained by a polynomial of degree $d$
in the number of edges of $G$ (Proposition \ref{bounded growth}).  Furthermore, if we fix a graph and modify it by either
subdividing edges or ``sprouting'' new leaves at a fixed set of vertices, then the dimension of $M$ evaluated on the modified graph
behaves as a polynomial of degree at most $d$ in the subdivision and sprouting parameters (Corollaries \ref{actual polynomial-sub}
and \ref{actual polynomial-spr}).

Sometimes we have no control of the generation degree of a finitely generated module, but we can still control its growth.
We say that $M$ is {\bf \boldmath{$d$}-small} if it is a subquotient of a module that is finitely generated in degrees $\leq d$, and 
{\bf \boldmath{$d$}-smallish} if it admits a filtration whose associated graded is $d$-small.  Theorem \ref{noetherian} implies that
$d$-small modules are finitely generated, and it is not hard to prove that the same is true for $d$-smallish modules (Proposition \ref{smallish fg}).
The degree of generation of such modules may be much larger than $d$, but for the purposes of the results mentioned in the
previous paragraphs, they grow as if they were finitely generated in degrees $\leq d$.  This will be important for the two classes
of examples that we study in detail, which we describe below.

\subsection{Homology of configuration spaces}\label{sec:introhom}
Given a graph $G$ and a positive integer $n$, the \textbf{\boldmath{$n$}-stranded unordered configuration space of \boldmath{$G$}} 
is the topological space
\[
\UConf_n(G) := \big\{(x_1,\ldots,x_n) \in G^n \bigmid x_i \neq x_j\big\}\big{/}S_n.
\]
The homology groups of these spaces have been extensively studied in settings both theoretical \cite{ADK,A,KP} and applied \cite{F}.

One powerful technique for studying these groups, which is applied for example in \cite{ADK}, is to fix the graph $G$ and consider the direct
sum of the homology groups of $\UConf_n(G)$ for all $n$.  This direct sum is a module over a polynomial ring with generators indexed by the edges
of $G$, where a variable acts by ``adding a point'' to the corresponding edge.  An orthogonal approach is to fix $n$ and vary $G$.
This approach has been used in a number of recent works \cite{RW,R,L,PR-trees}, and it is the approach that we take here.
In particular, the homology of $\UConf_n(G)$ is functorial with respect to contractions (Section \ref{sec:func}), 
and therefore defines an object of $\Rep_{\Z}(\cGgop)$.

\begin{theorem}\label{treefing}
Fix natural numbers $g$, $i$, and $n$. The $\cGopg$-module
\[
G \mapsto H_i\big(\UConf_n(G); \Z\big)
\]
is $(g+i+n)$-small.  In particular, it is finitely generated.
\end{theorem}

One concrete consequence of Theorems \ref{noetherian} and \ref{treefing} is that we obtain some control of the type of torsion that can appear in these homology groups.
We know from the work of Ko and Park that the only torsion that can appear in $H_1\big(\UConf_n(G); \Z\big)$ is 2-torsion \cite[Corollary 3.6]{KP}.
Furthermore, this torsion carries extremely interesting information: it is trivial if and only if $G$ is planar!
The topological meaning of torsion in higher degree homology is more mysterious, but we can at least show that there is a bound on the type
of torsion that can occur.

\begin{corollary}\label{torsion-cor}
For any triple $(g,i,n)$ of positive integers, there exists a constant $d_{g,i,n}$ such that for every graph $G$ of genus $g$, the torsion part of
$H_i\big(\UConf_n(G); \Z\big)$ has exponent at most $d_{g,i,n}$.
\end{corollary}

\begin{remark}
In this work we only consider \textit{unordered} configurations of points, mainly because the tools we use largely derive from the paper \cite{ADK}
and this is the setting in which they work. It is likely that one can obtain analogues of Theorem \ref{treefing} and Corollary \ref{torsion-cor}
for \textit{ordered} configuration spaces, starting by reproving certain results 
from \cite{ADK} in the ordered setting. 
\end{remark}

\subsection{Kazhdan--Lusztig coefficients}
Kazhdan--Lusztig polynomials of matroids are analogues of Kazhdan--Lusztig polynomials of Coxeter groups.
Just as Kazhdan--Lusztig polynomials of Weyl groups can be interpreted as Poincar\'e polynomials of certain intersection homology groups,
the same is true of Kazhdan--Lusztig polynomials of graphical (or, more generally, realizable) matroids.
See \cite{KLS} for a survey that explores this analogy in depth.

More precisely, given a graph $G$, we can define a complex variety $X_G$, called the {\bf reciprocal plane}, 
with the property that the coefficient of $t^i$ in the Kazhdan--Lusztig
polynomial of $G$ is equal to the dimension of $\IH_{2i}(X_G)$.  These homology groups are functorial with respect to contractions \cite{fs-braid},
thus we obtain an object of $\Rep_\mathbb{C}(\cGgop)$.

\begin{theorem}\label{kl-main}
Fix a natural number $g$ and a positive integer $i$.
The $\cGopg$-module
\[
G \mapsto \IH_{2i}(X_G)
\]
is $(2i-1+g)$-smallish.  In particular, it is finitely generated.
\end{theorem}

For example, Theorem \ref{kl-main} combines
with the results on subdivision described in Section \ref{intro-growth}
to imply that the $\ith$ Kazhdan--Lusztig polynomial of the matroid associated with the $n$-cycle
is a polynomial in $n$ of degree at most $i$.  Indeed, the formulas for these coefficients appearing in \cite{PWY} demonstrate
that this bound is sharp (Example \ref{cycle}).

\subsection{Outer automorphisms of free groups}
A further motivation for studying the category $\cGg$ and its modules is that this category
is closely related to $\Out(F_g)$, the outer automorphism group of a free group on $g$ generators.
This group is in many ways analogous to various arithmetic groups and to mapping class groups of surfaces,
and much work has gone into exploring its cohomology; see Vogtmann's ICM address \cite{Vogtmann-ICM} for a survey.

We call a graph $G$ of genus $g\geq 2$ {\bf reduced} if it has no bridges and no vertices of valence 2.
If we consider the full subcategory of $\cGg$ consisting of reduced graphs and replace it with an equivalent small
category, we obtain a category whose nerve is a classifying space for $\Out(F_g)$ (Corollary \ref{Kp1}).
This observation leads to the following theorem.

\begin{theorem}\label{ext}
Fix a non-negative integer $g$ and a commutative ring $k$.  Let $M\in\Rep_k(\cGgop)$ be the module that
assigns $k$ to every reduced graph and $0$ to every non-reduced graph,
with all nontrivial transition functions equal to the identity.  Then there is a canonical $k$-algebra isomorphism
$$\Ext^*_{\Rep_k(\cGgop)}(M,M) \;\cong\; H^*(\Out(F_g); k).$$
\end{theorem}

Our proof of Theorem \ref{ext} relies on the very non-trivial theorem of
Culler and Vogtmann that outer space is contractible \cite{CV}.  Since $\Out(F_g)$ acts on outer space with finite stabilizers,
the rational cohomology of the quotient coincides with the rational cohomology of $\Out(F_g)$.  
We stress, however, that Theorem \ref{ext} holds for arbitrary coefficients. 

\excise{As a sample application of this Theorem \ref{ext}, we compute the first cohomology of 
$\Out(F_2)\cong\operatorname{GL}(2,\Z)$ with coefficients in an arbitrary field $k$ (Section \ref{sec:sample}),
with close attention to the dependence of the answer on the characteristic of $k$.
This calculation can be performed by other means (Remark \ref{direct}), and the result is surely not new to experts.
Still, we hope that it serves as an example of how one can use the category $\Rep_k(\cGgop)$ to perform non-trivial calculations of cohomology groups of $\Out(F_g)$ in arbitrary characteristic.}

\subsection{Relationship to earlier and later work}
We briefly address the relationship between this paper and the two related works \cite{PR-trees} and \cite{MPR}.
\begin{itemize}
\item
This paper generalizes the authors' previous paper \cite{PR-trees}, in which we prove Theorems \ref{noetherian}
and \ref{treefing} for the category $\cG_0$ of trees.  
The proof of Theorem \ref{noetherian} takes the argument used in \cite{PR-trees} as a starting point
and builds on this argument in order to treat graphs of higher genus.  While the idea of applying the techniques
of \cite{sam} is the same, there is a significant additional layer of technical difficulty in the higher genus setting.

Once we have established Theorem \ref{noetherian}, the proof of Theorem \ref{treefing} for arbitrary genus
is nearly identical to the proof in the genus 0 case.  Theorem \ref{kl-main} has no direct analogue in the genus 0 setting
because Kazhdan--Lusztig polynomials of trees are trivial.  The same goes for Theorem \ref{ext} because $\Out(F_0)$
is the trivial group.

\item
A more recent preprint of Miyata and the authors \cite{MPR}
deals with modules over a category $\cG$ whose objects are graphs of arbitrary genus and whose morphisms are built out
of contractions and edge deletions.  In particular, the category $\cG_g$ is the full subcategory of $\cG$ consisting of graphs
of genus $g$.  In that paper, we improve upon Theorems \ref{noetherian} and \ref{treefing} and Corollary \ref{torsion-cor}
by proving analogous results for the category $\cG$.  The proof of the analogue of Theorem \ref{noetherian} in that paper
also uses the machinery of \cite{sam}, but employs a different approach that is not based on choosing a spanning tree.
We believe that the approach based on spanning trees is better adapted to studying the Hilbert series of a module,
as outlined in \cite[Section 1.3]{Ramos-Hilbert}.

Theorem \ref{kl-main} has no analogue in \cite{MPR} because the intersection homology groups of the reciprocal plane
are not functorial with respect to deletions.  There is also nothing in that paper about automorphism groups of free groups.
\end{itemize}

\vspace{\baselineskip}
\noindent
{\em Acknowledgments:}
NP is supported by NSF grant DMS-1565036.  ER is supported
by NSF grant DMS-1704811. The authors would like to thank Melody Chan, Jim Davis, Dan Dugger, 
Steven Sam, Paul Seymour, Dev Sinha, and Karen Vogtmann for valuable conversations.

\section{Graph categories}
We begin by fixing terminology and conventions about graphs and trees and defining all of the various categories of decorated graphs
with which we will work in this paper.  The reader may want to skim this section at first and refer back to it as needed.

\subsection{Graphs}\label{sec:main defs}
By a {\bf graph}, we will mean a finite CW complex of dimension at most 1.
The 0-cells are called {\bf vertices} and the 1-cells are called {\bf edges}.  We will write $|G|$
for the number of edges of $G$.  If $G$ is a non-empty connected graph, we define the {\bf genus}
of $G$ to be the rank of the first homology group, or equivalently the number of edges minus the number of vertices plus 1.
If we refer to a graph of genus $g$, we will always implicitly mean that the graph is non-empty and connected.

If $f:G\to G'$ is a map of CW complexes, we say that $f$ is {\bf very cellular} if it takes every vertex to a vertex
and every edge to either a vertex or an edge.  An edge that maps to a vertex will be called a {\bf contracted edge}.
If $G$ and $G'$ are graphs, we define a {\bf graph morphism} from $G$ to $G'$ to be an equivalence class of very cellular maps,
where two very cellular maps are equivalent if and only if they are homotopic through very cellular maps.
We note that a graph morphism $\varphi:G\to G'$ induces a well defined map on vertex sets, and it also makes sense to talk
about the set of edges that are contracted by $\varphi$.

We define a {\bf smooshing} to be a surjective graph morphism with connected fibers,
and we define a {\bf contraction} to be a smooshing with contractible fibers.
In particular, any automorphism of $G$ is a contraction from $G$ to itself, which necessarily has no contracted edges.
More generally, a contraction is a smooshing between two graphs of the same genus.
We denote by $\cGg$ the category whose objects are graphs of genus $g$ and whose morphisms are contractions.

\subsection{Trees}\label{sec:trees}
The definitions in Sections \ref{sec:trees} and \ref{sec:rigid} will be used only in Section \ref{repcat}, where we prove Theorem \ref{noetherian}.

A {\bf tree} is a graph of genus 0, and a {\bf rooted tree} is a pair consisting of a tree and a vertex, which is called the {\bf root}.
The vertex set of a rooted tree is equipped with a natural partial order in which $v\leq w$ if and only if the unique path from $v$ to the root
passes through $w$ (so the root is maximal).  
A {\bf leaf} of a rooted tree is a minimal vertex with respect to this partial order.  

For any vertex $v$, we define a {\bf descendant} of $v$ to be a vertex covered by $v$ in the partial order.
A {\bf planar rooted tree} is a rooted tree along with a linear order on the set of descendants of each vertex $v$.
This induces a depth-first linear order on the entire vertex set of the tree.
A {\bf contraction of rooted trees} is a contraction of trees that preserves the root, and a {\bf contraction of planar rooted trees}
is a contraction of rooted trees with the additional property that,
if $v$ comes before $w$ in the depth-first order, then the first vertex in the preimage of $v$
comes before the first vertex in the preimage of $w$.
Let $\cRT$ and $\cPT$ be the contraction categories of rooted trees and planar rooted trees, respectively.

\begin{remark}\label{barter}
Barter \cite{Barter} defines the category $\RTi$ whose objects are rooted trees and whose morphisms 
are pointed order embeddings on vertex sets, 
along with the category $\PTi$ whose objects are planar rooted trees and whose morphisms are pointed
order embeddings that preserve the depth-first linear order.
In \cite[Proposition 2.4]{PR-trees}, we prove that $\RTi$ is equivalent to $\cRT^\op$, 
and a similar argument shows that $\PTi$ is equivalent to $\cPT^\op$.
We will make use of Barter's work, via these equivalences, in Section \ref{repcat}.
\end{remark}

Finally, we will need a labeled version of the above definitions.  Let $S$ be a finite set.
We define an {\bf \boldmath{$S$}-labeled planar rooted tree} to be a triple $(T,v,\ell)$, where
$(T,v)$ is a planar rooted tree and $\ell$ is a function from the set of vertices of $T$ to $S$.
The most naive way to define a contraction $\varphi:(T,v,\ell)\to(T',v',\ell')$ of labeled planar rooted trees would be 
to say that it is a contraction of planar rooted trees with the property that the pullback of $\ell'$ along $\varphi$ is equal to $\ell$.
This, however, is not quite what we want.  
If $\varphi:(T,v)\to (T',v')$ is a contraction of planar rooted trees and $\varphi^*:(T',v')\to(T,v)$ is the corresponding
pointed order embedding under the equivalence of Remark \ref{barter}, we want to impose the condition that the pullback
of $\ell$ along $\varphi^*$ is equal to $\ell'$.  The proof of \cite[Proposition 2.4]{PR-trees} tells us that 
$\varphi^*(w') = \max\varphi^{-1}(w')$, so the appropriate condition for $\varphi:(T,v,\ell)\to(T',v',\ell')$ to be 
an $S$-labeled contraction is that $\ell'(w') = \ell(\max\varphi^{-1}(w'))$ for all $w'\in T'$.
Equivalently, we say that a vertex $w$ of $T$ is {\bf \boldmath{$\varphi$}-maximal} if $u\leq w$ for all vertices $u$ with $\varphi(u)=\varphi(w)$,
and we say that $\varphi$ is an $S$-labeled contraction if and only if $\ell'\circ\varphi(w) = \ell(w)$ for all $\varphi$-maximal vertices $w$.

\subsection{Rigidified graphs}\label{sec:rigid}
If $G$ is a graph, a {\bf spanning tree} of $G$ is a contractible sub-complex of $G$ containing all of the vertices.
A {\bf rigidified graph} of genus $g$ is a graph of genus $g$ along with a choice of spanning tree 
and an ordering and orientation of the $g$ {\bf extra edges} that are not in the spanning tree.  More formally, 
fix once and for all a graph $R_g$ with one vertex and $g$ loops, called the {\bf rose of genus \boldmath{$g$}}.  Then a planar rooted graph of genus $g$ is a 
quadruple $(G,T,v,\tau)$, where $G$ is a graph of genus $g$, $(T,v)$ is a planar rooted spanning tree of $G$,
and $\tau$ is a graph isomorphism from $R_g$ to the quotient space $G/T$.

We denote by $\cPGg$ the category whose objects are rigidified graphs
of genus $g$ and whose morphisms are contractions that restrict to contractions of planar rooted trees
(in particular, only edges in the spanning tree can be contracted) 
and are compatible with the order and orientations of the extra edges.
We use the letter P in the notation because $\mathcal{PG}_0 \cong \cPT$.
The point of this definition is that rigid graphs are graphs with just enough extra structure to eliminate all nontrivial automorphisms.

\subsection{Reduced graphs}\label{sec:reduced}
Most of the definitions in Sections \ref{sec:reduced} and \ref{sec:marked} will be used only in Section \ref{sec:oc}, where we discuss connections
to outer automorphism groups of free groups.  The one exception is that the notion of a half-edge also appears in Section \ref{sec:cc}.

Fix a graph $G$.
A {\bf half-edge} of $G$ is defined to be an end, in the sense of \cite{ends}, of the relative interior of an edge.
For any half-edge $h$, there is an associated edge $e(h)$ and a vertex $v(h)$ that is incident to $e(h)$.
For any pair $(e,v)$ consisting of an edge and a vertex incident to that edge, there are either one or two half-edges $h$
with $e(h)=e$ and $v(h)=v$, depending on whether or not $e$ is a loop.
For any vertex $v$, the {\bf valence} of $v$ is defined to be the number of half-edges $h$ with $v(h) = v$.

An edge of $G$ is called a {\bf bridge} if deleting the edge increases the number of connected components.
We call a non-empty connected graph with no bridges and no vertices of valence 2 {\bf reduced}.  We also define the unique graph with one
vertex and one edge to be reduced, even though the vertex has valence 2.  
Intuitively, the idea is that any non-empty connected graph may be obtained from a reduced graph by subdividing
edges and ``uncontracting'' bridges, and there are finitely many isomorphism classes of reduced graphs of any fixed genus.
For example, there are two reduced graphs of genus 2 up to isomorphism, namely the rose $R_2 = \infty$ and the melon $\banana$.

\begin{remark}
If $G$ is reduced and $\varphi:G\to G'$ is a contraction, then $G'$ is also reduced.
For example, all contractions with domain equal to the melon are either automorphisms or maps to the rose, 
and all contractions with domain equal to the rose are automorphisms.
\end{remark}

We define $\cGgred$ to be the full subcategory of $\cGg$ whose objects are reduced graphs.
In the next section, we will want to talk about the nerve of this category, but one can only define the nerve of a small category.
For this reason, we choose a list $G_1,\ldots,G_r$ that includes a unique representative of each isomorphism class of reduced graphs of genus $g$,
and we let $\cGgred^{\operatorname{small}}$ be the full subcategory of $\cGgred$ with objects $G_1,\ldots,G_r$.
Thus $\cGgred^{\operatorname{small}}$ is a small category that is equivalent to $\cGgred$.

\subsection{Marked reduced graphs}\label{sec:marked}
If $G$ is a graph of genus $g$, a {\bf marking} of $G$ is a homotopy class of homotophy equivalences from the rose $R_g$ to $G$.
(Note that a marking is not required to be a graph morphism.)  A {\bf marked graph} of genus $g$ is a pair $(G,f)$, where $G$ is
a graph of genus $g$ and $f$ is a marking of $G$.  We observe that the set of all markings of $G$ is a torsor for $\Out(F_g)$.
A {\bf contraction} from $(G,f)$ to $(G',f')$ is a contraction $\varphi:G\to G'$
such that $f' = f\circ \varphi$.  
We define {\bf outer category} $\cOg$ to be the category whose objects are marked reduced graphs
of genus $g$ and whose morphisms are contractions.  The group $\Out(F_g)$ acts on $\cOg$
in a natural way, fixing the graph but changing the marking.

As in Section \ref{sec:reduced}, we would like to define a small subcategory of $\cO_g$ that is equivalent to $\cOg$.
We will do this in two
subtlely different ways, which we now describe.
Recall that we have chosen representatives $G_1,\ldots,G_r$ of the isomorphism classes of reduced graphs of genus $g$.
Let $\cOgsm$ be the full subcategory of $\cOg$ consisting of objects of the form $(G_i,f)$ for some $i$ and any marking $f$
of $G_i$.
Note that there are still isomorphisms between distinct objects
of $\cOgsm$.  Specifically, if $f$ is a marking of $G$ and $\varphi:G\to G$ 
is a nontrivial automorphism of $G$, then $f$ and $f\circ\varphi$
are distinct markings of $G$ but $\varphi:(G,f)\to (G,\varphi\circ f)$ is an isomorphism.  To eliminate this phenomenon,
we choose for each $G_i$ a representative of each $\Aut(G_i)$ orbit in the set of markings of $G_i$, and we define
$\cOgti$ to be the subcategory of $\cOg$ generated by these objects.  Note that the natural inclusions
$$\cOgti\subset\cOgsm\subset\cOg$$ are both equivalences.  

\begin{example}\label{outer-1-cat}
There is only one reduced graph of genus 1 up to isomorphism, namely the cycle $R_1$.  A marking of $R_1$
is the same as an orientation of the loop.  The category $\cO_1^{\operatorname{small}}$ has two objects, 
related by the action $\Out(F_1) \cong S_2$,
corresponding to the two choices of marking of $R_1$.
Neither object has nontrivial automorphisms.
The category $\cO_1^{\operatorname{tiny}}$ has only one object, and it has no nontrivial automorphisms.
We discuss the nerves of these categories in Example \ref{Sinfty}.
\end{example}

The advantage of working with $\cOgsm$ is that the action of $\Out(F_g)$ on $\cOg$ restricts to an action on $\cOgsm$, 
where it acts freely on the set of objects.
The advantage of working with $\cOgti$ is that it is a poset category in the following sense.

\begin{proposition}\label{poset}
If $(G,f)$ and $(G',f')$ are objects of $\cOgti$, then $|\Mor_{\cOgti}\big((G,f), (G',f')\big)|\leq 1$.
Furthermore, if there exists a morphism in both directions, then $(G',f')=(G,f)$.
In particular, the set of objects of $\cOgti$ admits a poset structure with $(G',f')\leq (G,f)$ if and only if there
exists a morphism from $(G,f)$ to $(G',f')$.
\end{proposition}

\begin{proof}
If $g\leq 1$, the proposition is trivial, so we assume that $g \geq 2$.  We begin by proving the proposition when $G=G'$.
In this case, the proposition says that, if $\sigma$ is an automorphism of $G$ that is homotopic to the identity, then $\sigma$
must in fact be equal to the identity.  This is proved in \cite[Lemma 1]{Zim}.

Next we consider the case where $G \neq G'$. Suppose that $f$ is a marking of $G$ and $\varphi:G \rightarrow G'$ and $\psi:G \rightarrow G'$
are contractions with $\varphi \circ f = \psi \circ f$.  This implies that $\varphi$ is homotopic to $\psi$.
By \cite[Lemma 1.3]{SV}, $\varphi$ and $\psi$ differ by an automorphism $\sigma$ of $G'$.  Since $\varphi$ is homotopic to $\psi$,
$\sigma$ is homotopic to the identity, therefore $\sigma$ is equal to the identity by the previous paragraph.
Thus $\varphi=\psi$, as desired.
\end{proof}

\section{Local Noetherianity}\label{repcat}
The purpose of this section is to prove Theorem \ref{noetherian}, which says that $\Rep_k(\cGgop)$ is locally Noetherian
for any Noetherian commutative ring $k$.

\subsection{Gr\"obner theory of categories}

Let $\C$ be an essentially small category and $x$ an object of $C$. 
We define $\C_x$ to be the set of equivalence classes of morphisms out of $x$,
where $f\in\Mor_\C(x,y)$ is equivalent to $g\in\Mor_\C(x,y')$ if there exists an isomorphism $h$ from $y$ to $y'$ such that $h\circ f = g$.
The set $\C_x$ comes equipped with a natural quasi-order defined by putting 
\[
f \leq g \iff \text{ there exists a morphism $h$ with } h \circ f = g.
\]
Note that it is possible to have $f\leq g$ and $g\leq f$ even if the targets of $f$ and $g$ are not isomorphic, hence $\leq$ is only a quasi-order.
An infinite sequence $f_0, f_1, f_2, \ldots$ of elements of $\C_x$ is called {\bf bad} 
if there is no pair of indices $i < j$ such that $f_i \leq f_j$.
The category $\C$ is said to satisfy property {\bf (G2)} if, for every object $x$ of $\C$, $\C_x$ admits no bad sequences.
The category $\C$ is said to satisfy property {\bf (G1)} if, for every object $x$ of $\C$, $\C_x$ admits a linear order $\preceq$ that is compatible with post-composition
in the following sense:
if $f,g\in \Mor_\C(x,y)$, $h\in\Mor_\C(y,z)$, and $f\preceq g$, then $h \circ g \preceq h \circ f$.
The category $\C$ is called \textbf{Gr\"obner} 
if it satisfies properties (G1) and (G2) and has no endomorphisms other than the identity maps.

\begin{remark}
Sam and Snowden \cite{sam} explain that the motivation for properties (G1) and (G2) is deeply rooted in Gr\"obner basis theory 
from commutative algebra, with $\leq$ playing the role of the natural divisibility order on monomials and $\preceq$ 
playing the role of a term order such as the lexicographic order.
\end{remark}

Let $\C$ and $\C'$ be categories and let $\Phi:\C' \rightarrow \C$ be a functor. We say that $\Phi$ satisfies property \textbf{(F)} if, for all objects $x$ of $\C$, 
there exists a finite collection of objects $y_1,\ldots,y_r$ of $\C'$ and morphisms $f_i:x \rightarrow \Phi(y_i)$ such that, for any object $y$ of $\C'$ and any morphism $f:x \rightarrow \Phi(y)$, there exists a morphism $g:y_i \rightarrow y$ with $f = \Phi(g) \circ f_i$. 
We say $\C$ is \textbf{quasi-Gr\"obner} if there exists a Gr\"obner category $\C$ and an essentially surjective functor $\Phi:\C' \rightarrow \C$ satisfying property (F).

The motivation for these definitions comes from the following two theorems, both of which are of fundamental
importance in our work.

\begin{theorem}\label{fg}{\em \cite[Proposition 3.2.3]{sam}}
If $\Phi: \cC \rightarrow \cC'$ has property (F) and $M$ is a finitely generated $\cC'$-module,
then $\Phi^*M$ is a finitely generated $\cC$-module.
\end{theorem}

\begin{theorem}\label{repcatnoeth}{\em \cite[Theorem 1.1.3]{sam}}
If $\C$ is quasi-Gr\"obner and $k$ is a Noetherian commutative ring, then $\Rep_k(\C)$ is locally Noetherian.
\end{theorem}

\subsection{The category of rigidified graphs of fixed genus is Gr\"obner}
We begin with the following translation of Barter's work to our setting.

\begin{theorem}\label{grobnertrees}
The category $\cPT^\op \cong \mathcal{PG}_0^\op$ of planar rooted trees is Gr\"obner.
\end{theorem}

\begin{proof}
Barter proves that $\PTi$ is Gr\"obner \cite{Barter}, and the same is true
of $\cPT^\op$ by Remark \ref{barter}. 
\end{proof}

Our goal in this section is to extend Theorem \ref{grobnertrees} to the category $\cPGgop$ for arbitrary genus $g$.
We begin with the following corollary of Theorem \ref{grobnertrees}.

\begin{corollary}\label{G1}
For any natural number $g$, the category $\cPGgop$ satisfies property (G1).
\end{corollary}

\begin{proof}
Fix a rigidified graph $(G,T,v,\tau)$ of genus $g$.
We need to define a linear order $\preceq$ on equivalence classes of contractions of rigidified graphs with target $(G,T,v,\tau)$ 
(since we are working with the opposite category) that is compatible with pre-composition.
By Theorem \ref{grobnertrees}, we know that $\cPT^\op$ satisfies property (G1), which means that we have a linear
order on contractions of planar rooted trees with target $(T,v)$ which is compatible with pre-composition.  
Since a contraction of rigidified graphs restricts to a contraction of planar rooted trees, this induces a partial order on 
contractions of rigidified graphs with target $(G,T,v,\tau)$ that is compatible with pre-composition.
Let $\preceq$ be any linear refinement of this partial order.  We claim that $\preceq$ is also compatible with pre-composition.

To see this, suppose that $\varphi,\psi:(G',T',v',\tau')\to (G,T,v,\tau)$ are contractions with $\varphi\prec\psi$ 
and $\sigma:(G'',T'',v'',\tau'')\to (G',T',v',\tau')$
is an arbitrary contraction.  Since $\varphi\neq\psi$ and a contraction of rigidified graphs is determined by its restriction to
the spanning tree, the restrictions of $\varphi$ and $\psi$ to $(T',v')$ must be distinct.  This implies
that these restrictions are comparable in the Barter order, and therefore that the restrictions of $\varphi\circ\sigma$ and $\psi\circ\sigma$
to $(T'',v'')$ are comparable in the Barter order.  Since $\preceq$ refines the Barter order, we may conclude 
that $\varphi\circ\sigma\prec\psi\circ\sigma$.
\end{proof}

Our next task is to prove that $\cPGgop$ satisfies property (G2). 
We begin by stating a version of Kruskal's tree theorem for labeled planar rooted trees.
Let $S$ be a finite set.  If $(T,v,\ell)$ and $(T',v',\ell')$ are $S$-labeled planar rooted trees, we define $(T',v',\ell')\leq (T,v,\ell)$
if there exists a contraction from $(T,v,\ell)$ to $(T',v',\ell')$.  This defines a quasi-order on the set of isomorphism classes of $S$-labeled planar rooted trees.

\begin{theorem}\label{kruskal}
Let $S$ be a finite set.
The quasi-order on the set of isomorphism classes of $S$-labeled planar rooted trees admits no bad sequences.
\end{theorem}

\begin{proof}
After using Remark \ref{barter} to translate between order embeddings and contractions, the case where $S$ is a singleton
is proved in \cite[Lemma 10]{Barter}.  On the other hand, the theorem is proved for general $S$, but with rooted trees
instead of planar rooted trees, in \cite[Theorem 1.2]{Draisma}.  Both proofs are essentially the same, and are in fact modeled
on the original proof of Nash-Williams for unlabeled rooted trees \cite{Nash-Williams}.  These arguments can be trivially modified
to cover the result stated above.
\end{proof}

The following corollary is a relative version of Theorem \ref{kruskal}.
The case where $S$ is a singleton is proved in \cite[Theorem 9]{Barter}.  However, it turns out that the proof is greatly simplified
by allowing labels, as we demonstrate below.

\begin{corollary}\label{labeled G2}
Let $S$ be a finite set and let $(T,v,\ell)$ be an $S$-labeled planar rooted tree.
The set $(\cPT_{\!S}^\op)_{(T,v,\ell)}$ admits no bad sequences.
\end{corollary}

\begin{proof}
An element of $(\cPT_{\!S}^\op)_{(T,v,\ell)}$ is represented by a pair consisting of an $S$-labeled planar rooted tree $(T',v',\ell')$
and a contraction $\varphi':(T',v',\ell')\to (T,v,\ell)$.  Let $U := S \times \big(\Vert(T)\sqcup\{0\}\big)$ and
define a $U$-labeled planar rooted tree
$(T',v',\ell'_U)$ by putting $$\ell'_U(w') := 
\begin{cases}(\ell'(w'), \varphi'(w')) &\text{if $w'$ is $\varphi'$-maximal}\\
(\ell'(w'), 0) &\text{otherwise.}
\end{cases}
$$
Suppose that $\varphi':(T',v',\ell')\to (T,v,\ell)$ and $\varphi'':(T'',v'',\ell'')\to (T,v,\ell)$ represent two elements of $(\cPT_{\!S}^\op)_{(T,v,\ell)}$
and let $(T',v',\ell'_U)$ and $(T'',v'',\ell''_U)$ be the corresponding $U$-labeled planar rooted trees.
We have $\varphi'\leq\varphi''$
with respect to the quasi-order on $(\cPT_{\!S}^\op)_{(T,v,\ell)}$ if and only if there exists an $S$-labeled 
contraction $\psi:(T'',v'',\ell'')\to (T',v',\ell')$ such that $\varphi''=\varphi'\circ\psi$.
On the other hand, we have $(T',v',\ell'_U)\leq(T'',v'',\ell''_U)$
with respect to the quasi-order on isomorphism classes of $U$-labeled planar rooted trees
if and only if there exists a $U$-labeled 
contraction $\psi:(T'',v'',\ell''_U)\to (T',v',\ell'_U)$.

We claim that an $S$-labeled contraction $\psi$ is a $U$-labeled contraction if and only if $\varphi''=\varphi'\circ\psi$.
The easiest way to see this is to use Remark \ref{barter} to translate from contractions to pointed order embeddings,
as the statement becomes tautological in that setting.  This implies that 
any bad sequence in $(\cPT_{\!S}^\op)_{(T,v,\ell)}$ induces a bad sequence of isomorphism classes of $U$-labeled planar rooted trees,
and Theorem \ref{kruskal} tells us that no such sequences exist.
\end{proof}

Let $S = \{0,1\}^{2g}$.  Given a ridigified graph $(G,T,v,\tau)$ of genus $g$, we construct an $S$-labeled planar rooted graph $(T,v,\ell)$ as follows.
Recall that $\tau$ induces an ordering and an orientation on the $g$ extra edges of $G$.
For each $1\leq i\leq g$, let $w_{2i-1}$ be the vertex at which the $\ith$ extra edge originates and let $w_{2i}$ be the vertex at which 
the $\ith$ extra edge terminates.
Then for each vertex $w$ and each $1\leq j\leq 2g$, define the $j^\text{th}$ component of $\ell(w)$ to be 1 if $w\geq w_j$ and 0 otherwise.

\begin{lemma}\label{extra labels}
Let $(G,T,v,\tau)$ and $(G',T',v',\tau')$ be rigidified graphs of genus $g$, and let $(T,v,\ell)$ and $(T',v',\ell')$
be the associated $S$-labeled planar rooted graphs.  Let $\varphi:(T,v)\to (T',v')$ be a contraction of planar rooted graphs.
Then $\varphi$ induces a contraction of rigidified graphs if and only if it is compatible with the $S$-labeling.
\end{lemma}

\begin{proof}
On one hand, $\varphi$ induces a contraction of rigidified graphs if and only if $\varphi(w_j) = w'_j$ for all $j$.
On the other hand, $\varphi$ is compatible with the $S$-labeling if and only if, for all $\varphi$-maximal vertices $w$, $w \geq w_j\iff \varphi(w)\geq w'_j$.

Assume first that $\varphi$ induces a contraction of rigidified graphs, and let $w$ be a $\varphi$-maximal vertex.  If $w\geq w_j$,
then $\varphi(w) \geq \varphi(w_j) = w'_j$.  Conversely, if $\varphi(w)\geq w'_j$, then $w$ lies above
that unique $\varphi$-maximal preimage of $w'_j$, which in turn lies above $w_j$.

Assume next that $\varphi$ is compatible with the $S$-labeling.
For any $j$, we want to show that $\varphi(w_j) = w_j'$.  Since $w_j\geq w_j$, we know that $\varphi(w_j)\geq w_j'$.
To prove the opposite inequality, let $u_j$ be the unique $\varphi$-maximal preimage of $w_j'$.  Then
$$\varphi(u_j) = w_j' \Rightarrow u_j \geq w_j\Rightarrow w_j' = \varphi(u_j) \geq \varphi(w_j).$$
This completes the proof.
\end{proof}

\begin{corollary}\label{G2}
For any natural number $g$, the category $\cPGgop$ satisfies property (G2).
\end{corollary}

\begin{proof}
Fix a rigidified graph $(G,T,v,\tau)$ of genus $g$, and let $(T,v,\ell)$ be its associated $S$-labeled planar rooted graph.  
We need to prove that the set $(\cPGgop)_{(G,T,v,\tau)}$ admits no bad sequences.
By Lemma \ref{extra labels}, such a bad sequence induces a bad sequence in $(\cPT_{\!S}^\op)_{(T,v,\ell)}$, and Corollary \ref{labeled G2}
says that no such sequences exist.
\end{proof}

We are now ready to prove the main result of this section.

\begin{theorem}\label{planargraphgrob}
For any $g \geq 0$, the category $\cPGgop$ is Gr\"obner.
\end{theorem}

\begin{proof}
This follows from Corollaries \ref{G1} and \ref{G2}, along with the fact that rigidified graphs have no nontrivial automorphisms.
\end{proof}

\subsection{The category of graphs of fixed genus is quasi-Gr\"obner}

\begin{lemma}\label{F}
The forgetful functor $\Phi:\cPGgop\to\cGgop$ is essentially surjective and has property (F).
\end{lemma}

\begin{proof}
Essential surjectivity is clear.
For any genus $g$ graph $G$, we need to choose a finite collection of genus $g$ rigidified graphs $(G_i, T_i, v_i, \tau_i)$
along with contractions $\varphi_i:G_i\to G$ such that, for every genus $g$ rigidified graph $(G',T',v',\tau')$ and every contraction
$\varphi:G'\to G$, there exists an index $i$ and a contraction $\psi:(G',T',v',\tau')\to (G_i, T_i, v_i, \tau_i)$ such that $\varphi = \varphi_i\circ\psi$.

For our rigidified graphs $(G_i, T_i, v_i, \tau_i)$ and our contractions $\varphi_i$, we will choose a representative of 
every possible isomorphism class of such structures whose number of edges is at most $|G|+g$.
Since there is a finite number of rigidified graphs with a fixed number of edges and finitely many contractions between any two graphs,
there are only finitely many such choices.

Let $(G',T',v',\tau')$ and $\varphi$ be given, and let $E\subset\Edge(G')$ be the set of edges that are contracted by $\varphi$.
Let $\psi$ be the canonical contraction from $(G',T',v',\tau')$ to $(G'/(E\cap T'), T'/(E\cap T'), v', \tau')$.
It is clear from the definition that $\varphi$ factors through $\psi$.  It thus remains only to show that the number of edges of $G'/(E\cap T')$
is at most $|G|+g$.  Indeed, 
we have $|E| = |G'| - |G|$ and $|T'| = |G'| - g$, thus $|E\cap T'| \geq |G'| - (|G|+g)$.
From this it follows that $|G'/(E\cap T')| = |G'| - |E\cap T'| \leq |G| + g$.
\end{proof}

\begin{proof}[Proof of Theorem \ref{noetherian}.]
By Theorem \ref{planargraphgrob} and Lemma \ref{F}, $\cGgop$ is quasi-Gr\"obner.
Theorem \ref{noetherian} then follows from Theorem \ref{repcatnoeth}.
\end{proof}

\section{Smallness and growth}\label{sec:small}
We define what it means for a module over $\cGgop$ to be generated in low degree, and see what this tells us about its dimension growth.

\subsection{Generation degree, smallness, and smallishness}
Fix a Noetherian commutative ring $k$.
For any genus $g$ graph $G$, let $P_G\in \Rep_k(\cGgop)$ be the {\bf principal projective} module that assigns to a genus $g$ graph $G'$
the free $k$-module with basis $\Mor_{\cGg}(G',G)$.
Note that a module $M$ is finitely generated if and only if it is isomorphic to a quotient of a finite sum of principal projectives.
We say that a module $M\in\Rep_k(\cGgop)$ is
{\bf finitely generated in degree \boldmath{$\leq d$}} if we only need to use principal projectives 
corresponding to graphs with $d$ or fewer edges.
The following lemma illustrates this notion in a specific example.

\begin{lemma}\label{E}
Let $E$ be the $\cGopg$-module that takes a graph $G$ to the free $k$-module with basis indexed by edges of $G$, with the obvious maps.
The $\cGopg$-module $E^{\otimes i}$ is generated in degrees $\leq g+i$.
\end{lemma}

\begin{proof}
For any graph $G$ of genus $g$, $E^{\otimes i}(G)$ has a basis given by an ordered $i$-tuple of edges, and any such basis
element is in the image of the map induced by a contraction $\varphi:G\to G'$ if and only if none of the distinguished
edges are contracted by $\varphi$.  If $G$ has more than $g+i$ edges, then it has more than $i$ edges that are not loops,
therefore for any given $i$-tuple, one can find a non-distinguished edge to contract.
\end{proof}

We say that a module $M$ is {\bf \boldmath{$d$}-small} if it is a subquotient of a module that is generated in degrees $\leq d$.
We say that $M$ is {\bf \boldmath{$d$}-smallish} if it admits a filtration whose associated graded is $d$-small.

\begin{proposition}\label{smallish fg}
If $M$ is $d$-smallish for some $d$, then $M$ is finitely generated.
\end{proposition}

\begin{proof}
Choose a filtration of $M$ such that the associated graded $\gr M$ is $d$-small.
Theorem \ref{noetherian} implies that $\gr M$ is finitely generated.  This means that there is a finite collection
$G_1,\ldots,G_r$ of genus $g$ graphs, along with elements $v_i\in \gr M(G_i)$,
such that, for any genus $g$ graph $G$, the natural map
$$\bigoplus_{i=1}^r \bigoplus_{\varphi:G\to G_i} k\cdot e_{i,\varphi} \to \gr M(G)$$
taking $e_{i,\varphi}$ to $\varphi^* v_i$
is surjective.  For each $i$, choose an arbitrary lift $\tilde v_i\in M(G_i)$ of $v_i$.
Since surjectivity is an open condition, the nautral map
$$\bigoplus_{i=1}^r \bigoplus_{\varphi:G\to G_i} k\cdot e_{i,\varphi} \to M(G)$$
taking $e_{i,\varphi}$ to $\varphi^* \tilde v_i$
is also surjective, which means that $M$ is finitely generated.
\end{proof}

\begin{proposition}\label{bounded growth}
Let $k$ be a field, and suppose that $M\in\Rep_k(\cGgop)$ is $d$-smallish.
Then there exists a polynomial $f_M(t)\in\Z[t]$ of degree at most $d$ such that, for all $G$,
$\dim_k M(G) \leq f_M(|G|)$.
\end{proposition}

\begin{proof}
We may immediately reduce to the case where $M$ is the principal projective $P_{G'}$ for some genus $g$ graph $G'$ with $d$ edges.
For any $G$, a contraction from $G$ to $G'$ is determined, up to automorphisms of $G'$,
by a choice of $|G|-d$ edges of $G$ to contract.  The number of such choices is $\binom{|G|}{d}$, so
$\dim_k P_{G'}(G) \leq |\operatorname{Aut}(G')|\binom{|G|}{d}$.
\end{proof}

\subsection{Subdivision}\label{sec:subdivision}
Fix a graph $G$ of genus $g$, a natural number $r$, and an ordered $r$-tuple $\ue = (e_1,\ldots,e_r)$ of distinct directed
non-loop edges of $G$.  For any ordered $r$-tuple $\um = (m_1,\ldots,m_r)$ of natural numbers, let $G(\ue,\um)$ be 
the tree obtained from $G$ by subdividing each edge $e_i$ into $m_i$ edges.  
The number $m_i$ is allowed to be zero, and we adopt the convention
that subdividing $e_i$ into 0 edges means contracting $e_i$.  
For each $i$, the graph $G(\ue,\um)$ has a directed path of length $m_i$ where the directed edge $e_i$ used to be, 
and we label the vertices of that path $v_i^0,\ldots,v_i^{m_i}$.

Let $\OI$ be the category whose objects are linearly ordered finite sets and whose morphisms are ordered inclusions.
Every object of $\OI$ is isomorphic via a unique isomorphism to the finite set $[m]$ for some $m\in\N$.
For any $\um\in\N^r$, let $[\um]$ denote the corresponding object of the product category $\OIr$.

Our goal in this section is to define a {\bf subdivision functor} $\Phi_{G,\ue}:\OIr\to\cGgop$ and prove that $\Phi_{G,\ue}$ 
has property (F).  
We define our functor on objects by putting $\Phi_{G,\ue}([\um]) := G(\ue,\um)$.
Let $\uf = (f_1,\ldots,f_r)$ be a morphism in $\OIr$ from $[\um]$ to $[\un]$.  
We define the corresponding contraction $$\Phi_{G,\ue}(\uf):G(\ue,\un)\to G(\ue,\um)$$
by sending $v_i^t$ to $v_i^s$, where $s$ is the maximal element of the set $\{0\}\cup \{j\mid f_i(j)\leq t\} \subset \{0,1,\ldots,m_i\}$.

For any $\un \in\N^r$, let $|\un| := \sum n_i$.
We say that a contraction $\varphi:G(\ue,\un)\to G'$ 
{\bf factors nontrivially} if there exists a non-identity morphism $\uf:[\um]\to[\un]$ in $\OIr$
and a contraction $\psi:G(\ue,\um)\to G'$ such that $\varphi = \psi\circ\Phi_{G,\ue}(\uf)$.

\begin{proposition}\label{OIrF}
The subdivision functor $\Phi_{G,\ue}:\OIr\to\cGgop$ has property (F).
\end{proposition}

\begin{proof}
Property (F) says exactly that, for any graph $G'$ of genus $g$, the set of contractions 
from some $G(\ue,\um)$ to $G'$ that do not factor nontrivially is finite.
Let $\varphi:G(\ue,\um)\to G'$ be given.
We have $$|G(\ue,\um)| = |G| + |\um| - r,$$ so $\varphi$ must contract $|G| + |\um| - r - |G'|$ edges.
If $|\um|$ is sufficiently large, then at least one of those edges must be one of the subdivided edges.
We may then factor $\varphi$ nontrivially by first contracting that edge.

This tells us that, if we are looking for contractions from some $G(\ue,\um)$ to $G'$ that do not factor nontrivially,
we only need to consider finitely many $r$-tuples $\um$.  The proposition then
follows from the fact that all Hom sets in $\cGgop$ are finite.
\end{proof}

Proposition \ref{bounded growth} implies that the dimension of $M(G(\ue,\um))$ is bounded by a polynomial in $\um$
of degree at most $d$.  The following corollary to Proposition \ref{OIrF} says that the dimension of $M(G(\ue,\um))$ is in fact equal
to a polynomial in $\um$ when each coordinate is sufficiently large.

\begin{corollary}\label{actual polynomial-sub}
Let $k$ be a field, and suppose that $M\in\Rep_k(\cGgop)$ is $d$-smallish.  Then there exists a multivariate polynomial $f_{M,G,\ue}(t_1,\ldots,t_r)$
of total degree at most $d$ such that, if $\um$ is sufficiently large in every coordinate, 
$$\dim_k M(G(\ue,\um)) = f_{M,G,\ue}(m_1,\ldots,m_r).$$
\end{corollary}

\begin{proof}
Proposition \ref{smallish fg} tells us that $M$ is finitely generated,
though we have no control over the degree of generation.
Theorem \ref{fg} and Proposition \ref{OIrF} combine to tell us that $\Phi^*_{G,\ue}M$ is a finitely generated 
$\OIr$-module.  By \cite[Theorem 6.3.2, Proposition 6.3.3, and Theorem 7.1.2]{sam}, 
this implies that there exists a multivariate polynomial $f_{M,G,\ue}(t_1,\ldots,t_r)$
such that, if $\um\in\N^r$ is sufficiently large in every coordinate, 
$$\dim_k M(G(\ue,\um)) = \dim_k \Phi^*_{G,\ue}M([\um]) = f_{M,G,\ue}(m_1,\ldots,m_r).$$
Proposition \ref{bounded growth} says that $\dim_k M(G(\ue,\um))$ is bounded above by a polynomial of degree $d$ in 
the quantity $|G(\ue,\um)| = |G|-r+|\um|$, thus the total degree of $f_{M,G,\ue}(t_1,\ldots,t_r)$ can be at most $d$.
\end{proof}

\subsection{Sprouting}

Fix a graph $G$ of genus $g$, a natural number $r$, and an ordered $r$-tuple $\uv := (v_1,\ldots,v_r)$ of distinct vertices of $T$.
For any ordered $r$-tuple $\um = (m_1,\ldots,m_r)$ of natural numbers, let $G(\uv,\um)$ be 
the tree obtained from $G$ by attaching $m_i$ new edges to the vertex $v_i$, each of which has a new leaf as its other endpoint.
We will label the new leaves connected to the vertex $v_i$ by the symbols $v_i^1,\ldots,v_i^{m_i}$.

Our goal in this section is to define a {\bf sprouting functor} $\Psi_{G,\uv}:\OIr\to\cGgop$ and prove that $\Psi_{G,\uv}$ 
has property (F).  We define our functor on objects by putting $\Psi_{G,\ue}([\um]) := G(\uv,\um)$.
Let $\uf = (f_1,\ldots,f_r)$ be a morphism in $\OIr$ from $[\um]$ to $[\un]$.  
We define the corresponding contraction $$\Psi_{G,\uv}(\uf):T(\uv,\un)\to T(\uv,\um)$$
by fixing all of the vertices of $T$, sending $v_i^t$ to $v_i^s$ if $f_i(s) = t$, and sending $v_i^t$ to $v_i$ of $t$ is not in the image
of $f_i$.

As in Section \ref{sec:subdivision},
we say that a contraction $\varphi:G(\uv,\un)\to G'$ 
{\bf factors nontrivially} if there exists a non-identity morphism $\uf:[\um]\to[\un]$ in $\OIr$
and a contraction $\psi:G(\uv,\um)\to G'$ such that $\varphi = \psi\circ\Psi_{T,\uv}(\uf)$.

\begin{proposition}\label{OIrF2}
The sprouting functor $\Phi_{G,\uv}:\OIr \rightarrow \cGgop$ has property (F).
\end{proposition}

\begin{proof}
The philosophy of the proof is nearly identical to that of Proposition \ref{OIrF}.
Property (F) says exactly that, for any graph $G'$ of genus $g$, the set of contractions 
from some $G(\uv,\um)$ to $G'$ that do not factor nontrivially is finite.
Let $\psi:G(\uv,\um)\to G'$ be given.
We have $$|G(\uv,\um)| = |G| + |\um|,$$ so $\psi$ must contract $|G| + |\um| - |G'|$ edges.
If $|\um|$ is sufficiently large, then at least one of those edges must be one of the newly sprouted edges.
We may then factor $\psi$ nontrivially by first contracting that edge.

This tells us that, if we are looking for contractions from some $G(\ue,\um)$ to $G'$ that do not factor nontrivially,
we only need to consider finitely many $r$-tuples $\um$.  The proposition then
follows from the fact that all Hom sets in $\cGgop$ are finite.
\end{proof}

The proof of the following corollary is identical to the proof of Corollary \ref{actual polynomial-sub}, so we omit it.

\begin{corollary}\label{actual polynomial-spr}
Let $k$ be a field, and suppose that $M\in\Rep_k(\cGgop)$ is $d$-smallish.  
Then there exists a multivariate polynomial $f_{M,G,\uv}(t_1,\ldots,t_r)$
of total degree at most $d$ such that, if $\um$ is sufficiently large in every coordinate, 
$$\dim_k M(G(\uv,\um)) = f_{M,G,\uv}(m_1,\ldots,m_r).$$
\end{corollary}

\subsection{Combining small modules}
This section is devoted to stating and proving a lemma that we will need in Section \ref{kl-smallness}.

Let $H$ be a graph of genus $h$ with no loops.  For each vertex $v\in \Vert(H)$, fix a natural number $g_v$
and a $\cGop_{g_v}$-module $N_v$.  Let $g:= h + \sum_v g_v$.
Consider the $\cGopg$-module $N$ defined by putting
$$N(G) := \bigoplus_{\psi:G\to H} \bigotimes_{v\in \Vert(H)} N_v\big(\psi^{-1}(v)\big),$$
where the sum is over all smooshings $\psi:G\to H$ with the property that $\psi^{-1}(v)$ has genus $g_v$
for all $v$. 
If $\varphi:G\to G'$ is a contraction, the induced map $N(G')\to N(G)$ kills the $\psi$ summand
unless all of the edges contracted by $\varphi$ are also contracted by $\psi$.  If this is the case, then there is an induced
smooshing $\psi':G'\to H$ whose fibers are contractions of the fibers of $\psi$, and these contractions
induce a natural map from the $\psi$ summand
of $N(G)$ to the $\psi'$ summand of $N(G')$.

\begin{lemma}\label{graph together}
In the above situation, suppose that $N_v$ is $d_v$-small for all $v\in \Vert(H)$, and let $d := |H| + \sum_v d_v$.
Then the $\cGopg$-module $N$ is $d$-small.
\end{lemma}

\begin{proof}
We may immediately reduce to the case where each $N_v$ is a principal projective.
That is, for all $v\in \Vert(H)$, we have 
$N_v = P_{G_v}$ for some fixed graph $G_v$ of genus $g_v$
with $d_v$ edges.  The $k$-module $N(G)$ is spanned by classes indexed by tuples of maps
in which the set of edges that do {\em not} get contracted has cardinality $d$ and includes all of the loops.  
Thus, if $G$ has more than $d$ edges, at least one edge is a non-loop that does not get contracted,
which means that our class may be pulled back from some nontrivial contraction of $G$.
\end{proof}

\section{Homology of configuration spaces}\label{sec:cc}
The purpose of this section is to prove Theorem \ref{treefing}.
Our main technical tool is the reduced \'Swi\k{a}tkowski complex of An, Drummond-Cole, and Knudsen \cite{ADK}.
Sections \ref{sec:sw} and \ref{sec:func} are reproduced
from \cite[Sections 3.1 and 3.2]{PR-trees} for the reader's convenience.

\subsection{The reduced \'Swi\k{a}tkowski complex}\label{sec:sw}
Let $A_G$ be the integral polynomial ring generated by the edges of $G$. 
For any vertex $v$, let $S(v)$ denote the free $A_G$-module generated by the symbol $\emptyset$
along with all half-edges of $G$ with vertex $v$.
We equip $S(v)$ with a bigrading by defining an edge to have degree $(0,1)$, $\emptyset$ to have degree $(0,0)$,
and a half-edge to have degree $(1,1)$.  Let $\tS(v)\subset S(v)$ be the submodule generated by the elements
$\emptyset$ and $h - h'$ for all half-edges $h$ and $h'$.  We equip $\tS(v)$ with an $A_G$-linear differential $\partial_v$
of degree $(-1,0)$ by putting $$\partial (h-h') := \big(e(h)-e(h')\big) \emptyset  \and \partial\emptyset = 0.$$
We then define the
\textbf{reduced \'Swi\k{a}tkowski complex}
\[
\Sw(G) \;\;:=\;\; \bigotimes_{v \in \Vert(G)} \widetilde{S}(v),
\]
where the tensor product is taken over the ring $A_G$; this is a bigraded free $A_G$-module with a differential $\partial$.

For any graph $G$, let $H_\bullet\big(\UConf_\star(G)\big)$ denote the bigraded abelian group
\[
H_\bullet\big(\UConf_\star(G)\big) := \bigoplus_{(i,n)} H_i\big(\UConf_n(G); \Z\big).
\]

\begin{theorem}{\em \cite[Theorem 4.5 and Proposition 4.9]{ADK}}\label{swiacompute}
If $G$ has no isolated vertices, then
there is an isomorphism of bigraded abelian groups
\[
H_\bullet\big(\UConf_\star(G)\big) \cong H_\bullet\big(\Sw(G)\big).
\]
\end{theorem}

\begin{remark}
If $G$ is connected, then the only way that $G$ can have isolated vertices
is if $G$ is a single point.  In this case, $H_\bullet\big(\Sw(G)\big) = \Sw(G) = \Z$, concentrated in bidegree $(0,0)$,
whereas $H_\bullet\big(\UConf_\star(G)\big) \cong \Z \oplus \Z$, concentrated in bidegrees $(0,0)$ and $(0,1)$.
Thus the reduced \'Swi\k{a}tkowski complex fails only to recognize that the degree zero homology of $\UConf_1(G)$ is nontrivial.
\end{remark}

\subsection{Functoriality}\label{sec:func}
If $\varphi:G \rightarrow G'$ is contraction, then there is a natural map of differential bigraded modules
\[
\widetilde\varphi^\ast:\Sw(G') \rightarrow \Sw(G),
\]
which induces a map $$\varphi^\ast: H_i\big(\UConf_n(G'); \Z\big)\to H_i\big(\UConf_n(G); \Z\big)$$
by passing to homology \cite[Lemma C.7]{ADK}.
To describe $\widetilde\varphi^\ast$, we first consider the case where the number of edges of $G$ is one greater than the number of edges of $G'$;
we call such a contraction $\varphi$ a {\bf simple contraction}.
We identify the unique edge of $G$ that is contracted by $\varphi$ with the interval $[0,1]$.
Let $h_0$ (respectively $h_1$) be the half-edge of $G$ consisting of the vertex $0$ (respectively $1$) and the edge $[0,1]$.
Let $w'\in G'$ be the image of the edge $[0,1]$.
Each edge of $G'$ is mapped to isomorphically by a unique edge of $G$, and similarly for half-edges.  
This gives us a canonical ring homomorphism $A_{G'} \to A_G$ along with an $A_{G'}$-module homomorphism
$$\bigotimes_{v' \in \Vert(G')\smallsetminus\{w'\}} \widetilde{S}(v') \;\;\;\to \bigotimes_{v \in \Vert(G)\smallsetminus\{0,1\}} \widetilde{S}(v).$$
Given a half-edge $h'$ of $G'$ with $v(h') = w'$, let $h$ be the unique half-edge of $G$ mapping to $h'$.
We then define an $A_{G'}$-module homomorphism $$\widetilde{S}(w')\to \widetilde{S}(0)\otimes \widetilde{S}(1)$$
by the formula $$\emptyset\mapsto\emptyset\otimes\emptyset\and h'\mapsto
\begin{cases}
(h-h_0) \otimes \emptyset\;\;\;\text{if $v(h)=0$}\\
\emptyset \otimes (h-h_1)\;\;\;\text{if $v(h)=1$.}
\end{cases}$$
Tensoring these two maps together, we obtain the homomorphism $\widetilde\varphi^\ast:\Sw(G') \rightarrow \Sw(G)$,
and it is straightforward to check that this homomorphism respects the differential.
Arbitrary contractions may be obtained as compositions of simple contractions, and the induced homomorphism is independent of choice of factorization into simple contractions.
To summarize, we have the following result.

\begin{theorem}\label{summary}{\em \cite{ADK}}
There is a bigraded differential $\cGopg$-module that assigns to each graph $G$ the reduced \'Swi\k{a}tkowski complex $\Sw(G)$.
The homology of this bigraded differential $\cGopg$-module is the bigraded $\cGopg$-module that assigns to each graph $G$
the bigraded Abelian group $H_{\bullet}\big(\UConf_{\star}(G)\big)$.
\end{theorem}


\subsection{Smallness}
We are now ready to prove Theorem \ref{treefing} and Corollary \ref{torsion-cor}.

\begin{proof}[Proof of Theorem \ref{treefing}.]
Given a graph $G$ and a pair of natural numbers $i$ and $n$, let $\Sw(G)_{i,n}$ be the degree $(i,n)$ summand of the 
reduced \'Swi\k{a}tkowski complex. 
We will show that the $\cGopg$-module taking $G$ to the abelian group $\Sw(G)_{i,n}$
is generated in degrees $\leq g+i+n$.  Smallness will then follow from Theorem \ref{summary}.

The group $\Sw(G)_{i,n}$ is generated by elements of the form 
$$\sigma := e_1\cdots e_{n-i}\;\; \bigotimes_{j=1}^{i}\, (h_{j0} - h_{j1})\;\;\; 
\otimes \bigotimes_{v\notin\{v_1,\ldots,v_i\}} \!\!\!\emptyset,$$
where $e_1,\ldots,e_{n-i}$ are edges (not necessarily distinct), $v_1,\ldots,v_i$ are vertices (distinct), and, for each $j$,
$h_{j0}$ and $h_{j1}$ are half-edges at the vertex $v_j$.
For a particular $\sigma$ of this form, we will call $\{v_1,\ldots,v_i\}$ the set of {\bf distinguished vertices}.
Without loss of generality, we may assume that there is some integer $r$ with $0\leq r\leq i$ such that $v_j$
is adjacent to some distinguished vertex (possibly itself) if and only if $j\leq r$.  We may also assume that, if $j\leq r$,
$e(h_{j1})$ connects $v_j$ to some distinguished vertex (again, possibly $v_j$ itself).  If not, then $\sigma$ may be written as a difference
of classes of this form.

We call an edge $e$ a {\bf distinguished edge} if one of the following five conditions hold:
\begin{itemize}
\item $e$ is a loop
\item $e$ connects two distinguished vertices
\item $e = e_k$ for some $k\leq n-i$
\item $e = e(h_{j0})$ for some $j\leq i$
\item $e = e(h_{j1})$ for some $j\leq i$.
\end{itemize}
We will now argue that there are at most $g+i+n$ distinguished edges.
Let $t$ be the number of loops that are not at distinguished vertices.
Let $H$ be the induced subgraph on $\{v_1,\ldots,v_r\}$, which in particular contains all of the loops that are at distinguished vertices.
Since $H$ is a subgraph of $G$, and is missing $t$ loops, it has genus at most $g-t$, which means that it has at most
$r + g - t $ edges.  (Equality is achieved if and only if $r=0$ and $G$ is obtained by attaching $g$ loops to a tree,
in which case $H$ is empty and $t=g$.)
This means that the total number of distinguished edges is at most
$$t + (r + g - t) + (n-i) + i + (i-r) = g+i+n.$$

Let $G$ be given with $|G| > g+i+n$.  Since there are at most $g+i+n$ distinguished edges, we may choose an edge $e$
which is not distinguished.
Let $G' := G/e$ be the graph obtained from $G$ by contracting $e$, and let $\varphi:G\to G'$ be the canonical simple
contraction.  
Let $e_k'$ be the image of $e_k$ in $G'$, $v_j'$ the image of $v_j$ in $G'$, $h_{j0}'$
the image of $h_{j0}$ in $G'$, and $h_{j1}'$ the image of $h_{j1}$ in $G'$.  Let
$$\sigma' := e'_1\cdots e'_{n-i}\;\; \bigotimes_{j=1}^{i}\, (h_{j0}' - h_{j1}')\;\;\; 
\otimes \bigotimes_{v'\notin\{v'_1,\ldots,v'_i\}} \!\!\!\emptyset\;\;\;\in\;\;\;\Sw(G')_{i,n}.$$
We claim that $\sigma=\widetilde\varphi^*\sigma'$.

If $e$ is not incident to any vertex $v_j$, this is obvious.
The interesting case occurs when $e$ is incident to one of the distinguished
vertices.  Assume without loss of generality that it is incident to $v_1$, and let $w$ be the other end point of $e$.
Let $h$ be the half-edge of $T$ with $e(h) = e$ and $v(h) = v_1$ (this uniquely characterizes $h$ because $e$ is not a loop).
Applying the map $\varphi^*$ replaces each $e_k'$ with $e_k$.  When $j>1$, it replaces $h_{j0}'$ with $h_{j0}$
and $h_{j1}'$ with $h_{j1}$.  It replaces $h_{10}'$ with $h_{10} - h$ and $h_{11}'$ with $h_{11}-h$.
This means that it replaces $h_{j0}' - h_{j1}'$ with $h_{j0} - h_{j1}$, and therefore that 
$\widetilde\varphi^*\sigma' = \sigma$.

We thus conclude that every element of $\Sw(G)_{i,n}$ is a linear combination
of elements in the images of map associated with simple contractions; this completes the proof.
\end{proof}

\begin{proof}[Proof of Corollary \ref{torsion-cor}.]
Let $T_{g,i,n}\in\Rep_{\Z}(\cGgop)$ be the module that assigns to each graph $G$ the torsion subgroup of $H_i\big(\UConf_n(G); \Z\big)$.
By Theorem \ref{treefing}, $T_{g,i,n}$ is a submodule of a finitely generated module, and is therefore itself finitely generated.
We may then take $d_{g,i,n}$ to be the least common multiple of the exponents of the generators.
\end{proof}

\section{Kazhdan--Lusztig coefficients}\label{sec:kl}
For each $G$, let $R_G$ be the $\mathbb{C}$-subalgebra of rational functions in the variables $\{x_v\mid v\in\Vert(G)\}$
generated by the elements $\left\{\frac{1}{x_v-x_w}\Bigmid \text{$v\neq w$ adjacent}\right\}$, and let $X_G := \Spec R_G$.
The ring $R_G$ is called the {\bf Orlik-Terao algebra} of $G$ and the variety $X_G$ is called the {\bf reciprocal plane} of $G$.
We will be interested in the intersection homology group $\IH_{2i}(X_G)$ with coefficients in the complex numbers.

If $\varphi:G\to G'$ is a contraction, we obtain a canonical map from 
$\IH_{2i}(X_{G'})$ to $\IH_{2i}(X_G)$, and these maps compose in the expected way
\cite[Theorem 3.3(1,3)]{fs-braid}.  The purpose of this section is to study the $\cGop$-module
$\IH_{2i}$ that takes $G$ to $\IH_{2i}(X_G)$, and in particular to prove Theorem \ref{kl-main}.  

\subsection{Orlik-Solomon algebras}\label{sec:os}
For each $G$, let $\OS^\bullet(G)$ be the {\bf Orlik-Solomon algebra} \cite{OS80}
of the matroid associated with $G$ with coefficients in the complex numbers.
For any natural number $i$, we will denote the linear dual of $\OS^i(G)$ by $\OS_i(G)$.
For the purposes of this paper, we will need to know four things about the Orlik-Solomon algebra:
\begin{itemize}
\item $\OS^1(G)$ is spanned by classes $\{x_e\mid e\in\Edge(G)\}$, with relations $x_e = x_f$ if $e$ and $f$ are parallel
and $x_e = 0$ if $e$ is a loop.
\item $\OS^\bullet(G)$ is generated as a $\mathbb{C}$-algebra by $\OS^1(G)$.
\item If $G'$ is a contraction of $G$, we obtain a functorial map $\OS^\bullet(G)\to\OS^\bullet(G')$ 
by killing the generators indexed by contracted edges.  This in turn induces a map $\OS_\bullet(G')\to \OS_\bullet(G)$.
\item If $G$ is the disjoint union of $G_1$ and $G_2$, then $\OS^\bullet(G) \cong \OS^\bullet(G_1)\otimes \OS^\bullet(G_2)$.
\end{itemize}
By the third bullet point above, $\OS_i$ is a $\cGgop$-module for any natural number $i$.

\begin{lemma}\label{os-small}
For any natural number $i$, $\OS_i$ is $(g+i)$-small.
\end{lemma}

\begin{proof}
Recall from Lemma \ref{E} the $\cGgop$-module $E$ that assigns to any graph the $\mathbb{C}$-vector space with basis given by the edges.
By the first two bullet points above, $\OS^i(G)$ is a quotient of the $\ith$ tensor power of $E(G)^*$, therefore $\OS_i$ is a submodule of 
$E^{\otimes i}$.  Lemma \ref{E} says that $E^{\otimes i}$ is generated in degrees $\leq g+i$, therefore $\OS_i$ is $(g+i)$-small.
\end{proof}

\subsection{The spectral sequence}\label{sec:ss}
A subgraph $F\subset G$ with the same
vertex set is called a {\bf flat} of $G$
if its edge set is the set of contracted edges of a smooshing to a graph with no loops, which is denoted $G/F$.
Thus the vertex set of $G/F$ is identified with the set of connected components of $F$,
and the edge set is identified with $\Edge(G)\smallsetminus\Edge(F)$.
The {\bf rank} of $F$ is defined as the number of vertices
minus the number of connected components, and the {\bf corank} of $F$, denoted $\crk F$, is
the number of connected components minus 1.
The following theorem was proved in \cite[Theorems 3.1 and 3.3]{fs-braid}; see also \cite[Theorem 4.2]{PR-trees}.

\begin{theorem}\label{spectral}
For any graph $G$ and natural number $i$, there is a first quadrant homological spectral sequence $E(-,i)$
in the category of $\cGgop$-modules converging to $\IH_{2i}$,
with $$E(G,i)^1_{p,q} = \bigoplus_{\crk F = p}\OS_{2i-p-q}(F) \otimes \IH_{2(i-q)}(X_{G/F}).$$
If $\varphi:G\to G'$ is a contraction, the induced map 
$E(G',i)^1_{p,q}\to E(G,i)^1_{p,q}$ kills the $F$-summand unless $F$ contains all of the edges contracted by $\varphi$.  
In this case, the image of $F$
in $G'$ is a flat $F'$ of $G'$, and $G'/F'$ is canonically isomorphic to $G/F$.  The map takes
the $F$-summand of $E(G,i)^1_{p,q}$ to the $F'$-summand of $E(G',i)^1_{p,q}$ by 
the canonical map $\OS_{2i-p-q}(F)\to\OS_{2i-p-q}(F')$ tensored with the identity map on $\IH_{2(i-q)}(X_{G/F})$.
\end{theorem}


\subsection{Smallness}\label{kl-smallness}

\begin{proof}[Proof of Theorem \ref{kl-main}.]
By Theorem \ref{spectral}, $\IH_{2i}$ admits a filtration whose associated graded is isomorphic to the infinity page of $E(-,i)$,
therefore it is sufficient to show that, for all $p$ and $q$, $E(-,i)^1_{p,q}$ is
$(2i-1+g)$-small.

The set of flats of $G$ is in bijection with equivalence classes of smooshings with source $G$
for which the target has no loops, where two such smooshings
are equivalent if they differ by an automorphism of the target.  We therefore have
$$E(G,i)^1_{p,q} \;\cong \bigoplus_{|\Vert(H)| = p+1}\left( \bigoplus_{\substack{\psi:G\to H\\ \text{smooshing}}} 
\left(\bigotimes_{v\in\Vert(H)} \OS_*\big(\psi^{-1}(v)\big)\right)_{2i-p-q}
 \bigotimes \;\;\IH_{2(i-q)}(X_{H})\right)^{\Aut(H)}.$$
If we fix $H$ and require that the graph $\psi^{-1}(v)$ has genus $g_v$,
Lemmas \ref{graph together} and \ref{os-small} together imply that $E(-,i)^1_{p,q}$ is $d$-small, where $d = |H| + 2i-p-q + \sum_v g_v$.
If $h$ is the genus of $H$, then $|H| = p + h$ and $\sum_v g_v = g-h$, so $d = 2i + g - q$.  Since this is independent of the choice
of $H$ or of the numbers $g_v$, we can conclude that $E(-,i)^1_{p,q}$ is $(2i+g-q)$-small.

Finally, we note that $\IH_{2(i-q)}(X_{H}) = 0$ unless $2(i-q) < p$ or $q=i$ and $p=0$ \cite[Proposition 3.4]{EPW}, while $\OS_{2i-p-q}(F) = 0$
unless $p+q \leq 2i$.
In particular $E(-,i)^1_{p,0} = 0$ for all $p$, which implies that each $E(-,i)^1_{p,q}$ is
$(2i-1+g)$-small.
\end{proof}

\begin{remark}
The module $\IH_0 = E(-,0)^1_{0,0}$ is the constant module taking every graph to $\mathbb{C}$ and every morphism to the identity.
This module is $g$-small rather than $(g-1)$-small, which is why we required that $i$ be positive in the statement of Theorem \ref{kl-main}.
Indeed, one can see that the last sentence of the proof fails when $p=q=i=0$.
\end{remark}

\begin{example}\label{cycle}
When $g=1$, Theorem \ref{kl-main} and Corollary \ref{actual polynomial-sub} combine to say that the 
$\ith$ Kazhdan--Lusztig coefficient of the $n$-cycle should eventually agree with a polynomial in $n$
of degree at most $2i$.  In fact, it is equal to \cite[Theorem 1.2(1)]{PWY}
$$\frac{1}{i+1}\binom{n-i-2}{i}\binom{n}{i},$$ so our result is sharp.
\end{example}

\begin{example}
Let $G_g(a_1,\ldots,a_{g+1})$ be the genus $g>0$ graph obtained by taking the graph with two vertices and $g+1$ edges between
them and subdividing the $i^\text{th}$ edge into $a_i$ pieces.  Theorem \ref{kl-main} and Corollary \ref{actual polynomial-sub} say that the first Kazhdan--Lusztig
coefficient of $G_g(a_1,\ldots,a_{g+1})$ should eventually agree with a multivariate polynomial of total degree at most $g+1$ in $a_1,\ldots,a_{g+1}$.

The first Kazhdan--Lusztig
coefficient is equal to the number of corank 1 flats minus the number of rank 1 flats \cite[Proposition 2.12]{EPW}.
If $a_i > 1$ for all $i$, this is equal to
$$\prod_{i=1}^{g+1} a_i + \sum_{i=1}^{g+1} \binom{a_i}{2} - \sum_{i=1}^{g+1} a_i.$$
Thus our result is again sharp.
\end{example}

\section{Outer Category}\label{sec:oc}
The purpose of this section is to describe how one may use the category $\cGgred$ to compute cohomology groups
of $\Out(F_g)$ with arbitrary coefficients.

\subsection{Nerves of categories}
We begin by briefly reviewing some facts about small categories and their nerves. 
Let $\cC$ be a small category. Then we define the {\bf nerve} $|\cC|$ of $\cC$ to be the geometric realization
of the simplicial set defined as follows.  The $0$-simplicies are in bijection with the objects of $\cC$, while the $i$-simplicies 
for $i>0$ are in bijection with $i$-tuples of morphisms
\[
(f_1,\ldots,f_i)
\]
such that, for each $0 \leq j \leq i$, the codomain of $f_{j+1}$ agrees with the domain of $f_j$. For each $i > 0$ and $1 \leq j \leq i+1$ the face map $\partial_j$ is defined by
\[
\partial_j(f_1,\ldots,f_i) = \begin{cases} (f_2,\ldots,f_i) &\text{ if $j = 0$}\\ (f_1,\ldots,f_{i-1}) & \text{ if $j = i+1$}\\ (f_1,\ldots,f_{j-2},f_{j-1} \circ f_{j}, f_{j+1},\ldots ,f_i) \text{ otherwise.}\end{cases}
\]
The degeneracy map $\sigma_j$ is defined by
\[
\sigma_j(f_0,\ldots,f_i) = (f_0,\ldots, f_{j-1},\id,f_j,\ldots,f_i),
\]
where $\id$ is the identity map on the domain of $f_{j-1}$ (or the codomain of $f_j$). 

\begin{remark}\label{oppNerve}
We immediately see that there is a canonical homeomorphism $|\cC| \cong |\cC^\op|$.
A functor between two categories induces a map between their nerves, and an equivalence of categories
induces a homotopy equivalence between the nerves.
\end{remark}

Let $k$ be a commutative ring, and let 
$\uk\in\Rep_k(\C)$ be the module that takes every object to the 1-dimensional vector space $k$
and every morphism to the identity map.  The following standard result can be found, for example, in \cite[Theorem 5.3]{Webb}.

\begin{theorem}\label{nerve-ext}
There is a canonical graded $k$-algebra homomorphism $\Ext_{\Rep_k(\cC)}^*(\uk,\uk)\cong H^*(|\cC|; k)$.
\end{theorem}

\subsection{Outer category and the cohomology of \boldmath{$\Out(F_g)$}}
We begin with the following result, which relies heavily on Culler and Vogtman's work on outer space \cite{CV}.

\begin{theorem}\label{contract}
The nerves $|\cOgsm|$ and $|\cOgti|$ are contractible. 
\end{theorem}

\begin{proof}
The categories $\cOgsm$ and $\cOgti$ are equivalent, therefore Remark \ref{oppNerve} tells us that it is sufficient to prove
that $|\cOgti|$ is contractible.  By Proposition \ref{poset}, $\cOgti$ is a poset category, which implies that $|\cOgti|$ is homeomorphic
to the order complex of the poset structure on the set of objects.  This order complex is called the {\bf spine of outer space},
and it is known to be contractible \cite[Corollary 6.1.2 ]{CV}.
\end{proof}

Recall that we have an action of the group $\Out(F_g)$ on the category $\cOgsm$, which induces an action on the nerve.
We also have a functor $\Phi:\cOgsm\to\cGgred^{\operatorname{small}}$ given by forgetting the marking, 
and this functor induces a map
$\Phi_*:|\cOgsm|\to |\cGgred^{\operatorname{small}}|$ of nerves.

\begin{proposition}\label{you've got a lot of nerve}
The action of $\Out(F_g)$ on $|\cOgsm|$ is free and proper, and $\Phi_*:|\cOgsm|\to |\cGgred^{\operatorname{small}}|$
is the quotient map.
\end{proposition}

\begin{proof}
The fact that the action is free and proper follows from the fact that it is free on the set of objects (which correspond to 0-simplices)
and each group element acts by a simplicial map.  To see that $\Phi_*$ is the quotient map, we need to show that
it is surjective and its fibers coincide with the orbits of $\Out(F_g)$.  This follows from the fact that $\Out(F_g)$
acts transitively on the set of markings of a reduced graph of genus $g$.
\end{proof}

\begin{corollary}\label{Kp1}
The nerve $|\cGgred^{\operatorname{small}}|$ is a classifying space for the group $\Out(F_g)$.
\end{corollary}

\begin{example}\label{Sinfty}
Let us consider the very simple case where $g=1$, which we began discussing in Example \ref{outer-1-cat}.  
The category $\cO_1^{\operatorname{tiny}}$ has only one object (an oriented loop)
and no nontrivial morphisms, so its nerve is a point.  The category $\cO_1^{\operatorname{small}}$ has two objects, namely a loop
with two different orientations, and these two objects are uniquely isomorphic.  The nerve of $\cO_1^{\operatorname{small}}$ is an 
infinite-dimensional sphere $S^\infty$, and the group $\Out(F_1)\cong S_2$ acts via the antipodal map with quotient $\R P^\infty$.
The category $\cG_{1,\operatorname{red}}^{\operatorname{small}}$ has a single object with automorphism group $S_2$,
so its nerve is homeomorphic to $\R P^\infty$, which is a classifying space for $S_2$.
\end{example}

\begin{corollary}\label{small calc}
For any commutative ring $k$, we have $\Ext_{\Rep_k(\cGgred^\op)}^*(\uk,\uk)\cong H^*(\Out(F_g); k)$.
\end{corollary}

\begin{proof}
To compute $\Ext_{\Rep_k(\cGgred^\op)}^*(\uk,\uk)$, we may replace $\cGgred$ with the equivalent category
$\cGgred^{\operatorname{small}}$.  The result then follows from Remark \ref{oppNerve}, 
Theorem \ref{nerve-ext}, and Corollary \ref{Kp1}.
\end{proof}

\begin{proof}[Proof of Theorem \ref{ext}]
Given a pair of modules $M\in\Rep_k(\cGgop)$ and $N\in\Rep_k(\cGgred^\op)$, we will write $\overline{M}$
to denote the restriction of $M$ to $\Rep_k(\cGgred^\op)$ and $N^!$ to denote the extension of $N$ by zero to $\Rep_k(\cGgop)$.
The functors $M\mapsto\overline{M}$ and $N\mapsto N^!$ are exact and the former is left adjoint to the latter,
therefore $\Ext^*_{\Rep_k(\cGgred^\op)}(\overline{M}, N)\cong \Ext^*_{\Rep_k(\cGgop)}(M,N^!)$.
If we apply this fact with $M = \uk^!$ and $N = \uk$,
we see that Theorem \ref{ext} is equivalent to Corollary \ref{small calc}.
\end{proof}

\subsection{A sample calculation}\label{sec:sample}
We now use Corollary \ref{small calc} to compute the first cohomology of $\Out(F_2)\cong \operatorname{GL}(2; \Z)$
with coefficients in an arbitrary field $k$.  
In particular, we illustrate the extent to which
the representation theory of finite groups (namely automorphism groups of graphs) can be used to aid our calculations.

As in Section \ref{sec:reduced}, there are exactly two reduced graphs of genus 2 up to isomorphism, namely
the rose $\infty$ and the melon $\banana$.  The automorphism group of the rose is $D_4$, while the automorphism
group of the melon is $S_3\times S_2$.
Let $\varphi_1$, $\varphi_2$, and $\varphi_3$ be the three contractions from the melon to the rose obtained by cyclically
permuting the edges and then contracting the middle one.
Up to post-composition by an automorphism of the rose, every contraction is of this form.

Let $P_{\infty} \in \Rep_k((\cG_{2,\operatorname{red}}^{\operatorname{small}})^\op)$ be the principal projective module corresponding to the rose, 
and consider the surjection $P_{\infty}\to\uk$ that sends every basis element to 1.
Let $K$ be the kernel of this homomorphism.  Applying the functor $\Hom(-,\uk)$ gives us the long exact sequence
$$0\to \Hom(\uk,\uk)\to \Hom(P_{\infty},\uk)\to \Hom(K,\uk)\to \Ext^1(\uk,\uk)\to \Ext^1(P_\infty,\uk).$$
An element of $\Hom(P_{\infty},\uk)$ is determined by its value on the identity morphism of $\infty$, which implies that
the first map $ \Hom(\uk,\uk)\to \Hom(P_{\infty},\uk)$ is an isomorphism.  The fact that $P_\infty$ is projective
implies that $\Ext^1(P_\infty,\uk) = 0$, thus $\Hom(K,\uk)\to \Ext^1(\uk,\uk)$ must also be an isomorphism.
We therefore want to compute $\Hom(K,\uk)$.

An element of $\Hom(K,\uk)$ is a pair\footnote{Here
we are using the symbol $k$ to denote the 1-dimensional trivial representations of both $D_4$ and $S_3\times S_2$.}
$$(f,g)\in \Hom_{D_4}(K(\infty),k)\times \Hom_{S_3\times S_2}(K(\banana),k)$$
satisfying the condition that, if we pre-compose $g$ with any of the three
inclusions $K(\infty)\to K(\banana)$ induced by $\varphi_1$, $\varphi_2$, and $\varphi_3$, we obtain $f$.

Let's start by computing $\Hom_{S_3\times S_2}(K(\banana),k)$ and $\Hom_{D_4}(K(\infty),k)$.  The group $S_3\times S_2$ acts freely on the set of contractions
from the melon to the rose with two orbits, which we will call the {\bf untwisted contractions} and the {\bf twisted contractions}.
The untwisted contractions consist of the orbit that includes the three maps $\varphi_i$, and the twisted contractions consist of untwisted
contractions followed by an automorphism of the rose that fixes one of the two loops and reverses the orientation of the other loop.
We therefore have $P_\infty(\banana)\cong k[S_3\times S_2] \oplus k[S_3\times S_2]$ as representations of $S_3\times S_2$.
The space of homomorphisms from $P_\infty(\banana)$ to $k$ is 2-dimensional, with a basis given by the homomorphisms
that take the sum of the coefficients of the twisted or untwisted maps.
Applying $\Hom_{S_3\times S_2}(-,k)$ to 
the short exact sequence $0\to K(\banana)\to P_\infty(\banana)\to k\to 0$ and noting that $P_\infty(\banana)$ is a projective representation
of $S_3\times S_2$, we obtain the long exact sequence
$$0\to \Hom_{S_3\times S_2}(k,k)\to \Hom_{S_3\times S_2}(P_{\infty}(\banana),k)\to \Hom_{S_3\times S_2}(K(\banana),k)\to \Ext^1_{S_3\times S_2}(k,k)\to 0.$$
Since the abelianization of $S_3\times S_2$ is $S_2\times S_2$, we have $\dim \Ext^1_{S_3\times S_2}(k,k) = 2$ if $k$ has characteristic 2 and 0 otherwise.
Hence $\dim \Hom_{S_3\times S_2}(K(\banana),k) = 3$ if $k$ has characteristic 2 and 1 otherwise.  A similar argument for the rose tells us
that $\dim \Hom_{D_4}(K(\infty),k) = 2$ if $k$ has characteristic 2 and 0 otherwise.

Let's find explicit bases for our Hom spaces.  Let $h_1:K(\banana)\to k$ be the homomorphism that
adds the coefficients of the untwisted maps in $K(\banana)\subset P_\infty(\banana)$.  This homomorphism is well defined and nonzero for any field $k$.
Let $h_2:K(\banana)\to k$ be the homomorphism that adds the coefficients of $C_3\times S_2\subset S_3\times S_2$ for both the twisted
and untwisted maps and let 
$h_3:K(\banana)\to k$ be the homomorphism that adds the coefficients of $S_3\times \{\id\}\subset S_3\times S_2$ for both the twisted
and untwisted maps.  Each of these homomorphisms is well defined if and only if the characteristic of $k$ is 2, in which case it is straightforward to check
that $\{h_1,h_2,h_3\}$ is a basis for $\Hom_{S_3\times S_2}(K(\banana),k)$.
Let $f_1:K(\infty)\to k$ add the coefficients of the untwisted automorphisms of the rose (those generated by horizontal and vertical reflections),
and let $f_2:K(\infty)\to k$ add the coefficients of the automorphisms that keep the left loop on the left and the right loop on the right.
Each of these homomorphisms is well defined if and only if the characteristic of $k$ is 2, in which case it is straightforward to check
that $\{f_1,f_2\}$ is a basis for $\Hom_{S_3\times S_2}(K(\banana),k)$.

Finally, we observe that $h_1$ restricts to $f_1$ and $h_2$ restricts to $f_2$ under all three inclusions of $K(\infty)$ into $K(\banana)$.
On the other hand, the restriction of $h_3$ to $K(\infty)$ fails to be $D_4$-equivariant {\em and} depends on the choice of inclusion of $K(\infty)$ into $K(\banana)$.
We therefore conclude that
$$\dim H^1(\Out(F_2); k) = \begin{cases} 2 & \text{if $\operatorname{char}(k) =2$}\\ 0 & \text{otherwise}.\end{cases}$$

\begin{remark}\label{direct}
This result can also be obtained by working directly with a presentation for $\Out(F_2)$, such as the one in \cite[Section 2.1]{Vogt}.
This presentation can be used to compute the abelianization, and $H^1(\Out(F_2); k)$ is isomorphic to the vector space of group homomorphisms
from the abelianization to $k$.
\end{remark}

\bibliography{./symplectic}

\def\cprime{$'$}
\providecommand{\bysame}{\leavevmode\hbox to3em{\hrulefill}\thinspace}
\providecommand{\MR}{\relax\ifhmode\unskip\space\fi MR }
\providecommand{\MRhref}[2]{%
  \href{http://www.ams.org/mathscinet-getitem?mr=#1}{#2}
}
\providecommand{\href}[2]{#2}
\begin{thebibliography}{PWY16}

\bibitem[Abr00]{A}
Aaron~David Abrams, \emph{Configuration spaces and braid groups of graphs},
  2000, Thesis (Ph.D.)--University of California, Berkeley.

\bibitem[ADCK]{ADK}
Byung~Hee An, Gabriel~C. Drummond-Cole, and Ben Knudsen, \emph{Subdivisional
  spaces and graph braid groups}, \textsf{ arXiv:1708.02351}.

\bibitem[Bar]{Barter}
Daniel Barter, \emph{Noetherianity and rooted trees},
  \textsf{arXiv:1509.04228}.

\bibitem[CEF15]{CEF}
Thomas Church, Jordan~S. Ellenberg, and Benson Farb, \emph{F{I}-modules and
  stability for representations of symmetric groups}, Duke Math. J.
  \textbf{164} (2015), no.~9, 1833--1910.

\bibitem[CV86]{CV}
Marc Culler and Karen Vogtmann, \emph{Moduli of graphs and automorphisms of
  free groups}, Invent. Math. \textbf{84} (1986), no.~1, 91--119.

\bibitem[Dra14]{Draisma}
Jan Draisma, \emph{Noetherianity up to symmetry}, Combinatorial algebraic
  geometry, Lecture Notes in Math., vol. 2108, Springer, Cham, 2014,
  pp.~33--61.

\bibitem[EPW16]{EPW}
Ben Elias, Nicholas Proudfoot, and Max Wakefield, \emph{The {K}azhdan-{L}usztig
  polynomial of a matroid}, Adv. Math. \textbf{299} (2016), 36--70.

\bibitem[Far08]{F}
Michael Farber, \emph{Invitation to topological robotics}, Zurich Lectures in
  Advanced Mathematics, European Mathematical Society (EMS), Z\"{u}rich, 2008.

\bibitem[Fre31]{ends}
Hans Freudenthal, \emph{\"{U}ber die {E}nden topologischer {R}\"{a}ume und
  {G}ruppen}, Math. Z. \textbf{33} (1931), no.~1, 692--713.

\bibitem[KP12]{KP}
Ki~Hyoung Ko and Hyo~Won Park, \emph{Characteristics of graph braid groups},
  Discrete Comput. Geom. \textbf{48} (2012), no.~4, 915--963.

\bibitem[L{\"u}t]{L}
Daniel L{\"u}tgehetmann, \emph{Representation stability for configuration
  spaces of graphs}, \textsf{arXiv:1701.03490}.

\bibitem[MPR]{MPR}
Dane Miyata, Nicholas Proudfoot, and Eric Ramos, \emph{The categorical graph
  minor theorem}, \textsf{arXiv:2004.05544}.

\bibitem[NW63]{Nash-Williams}
C.~St. J.~A. Nash-Williams, \emph{On well-quasi-ordering finite trees}, Proc.
  Cambridge Philos. Soc. \textbf{59} (1963), 833--835.

\bibitem[OS80]{OS80}
Peter Orlik and Louis Solomon, \emph{Combinatorics and topology of complements
  of hyperplanes}, Invent. Math. \textbf{56} (1980), no.~2, 167--189.

\bibitem[PR19]{PR-trees}
Nicholas Proudfoot and Eric Ramos, \emph{Functorial invariants of trees and
  their cones}, Selecta Math. (N.S.) \textbf{25} (2019), no.~4, Paper No. 62,
  28.

\bibitem[Pro18]{KLS}
Nicholas Proudfoot, \emph{The algebraic geometry of
  {K}azhdan-{L}usztig-{S}tanley polynomials}, EMS Surv. Math. Sci. \textbf{5}
  (2018), no.~1, 99--127.

\bibitem[Put15]{Put}
Andrew Putman, \emph{Stability in the homology of congruence subgroups},
  Invent. Math. \textbf{202} (2015), 987--1027.

\bibitem[PWY16]{PWY}
Nicholas Proudfoot, Max Wakefield, and Ben Young, \emph{Intersection cohomology
  of the symmetric reciprocal plane}, J. Algebraic Combin. \textbf{43} (2016),
  no.~1, 129--138.

\bibitem[PY17]{fs-braid}
Nicholas Proudfoot and Ben Young, \emph{Configuration spaces, {$\rm
  FS^{op}$}-modules, and {K}azhdan-{L}usztig polynomials of braid matroids},
  New York J. Math. \textbf{23} (2017), 813--832.

\bibitem[Rama]{R}
Eric Ramos, \emph{An application of the theory of {FI}-algebras to graph
  configuration spaces}, \textsf{arXiv:1805.05316}.

\bibitem[Ramb]{Ramos-Hilbert}
\bysame, \emph{Hilbert series in the category of trees with contractions},
  \textsf{arXiv:2007.05669}.

\bibitem[RW]{RW}
Eric Ramos and Graham White, \emph{Families of nested graphs with compatible
  symmetric-group actions}, \textsf{arXiv:1711.07456}.

\bibitem[Sno13]{Sno}
Andrew Snowden, \emph{Syzygies of segre embeddings and $\delta $-modules}, Duke
  Mathematical Journal \textbf{162} (2013), 225--277.

\bibitem[SS17]{sam}
Steven~V. Sam and Andrew Snowden, \emph{Gr\"obner methods for representations
  of combinatorial categories}, J. Amer. Math. Soc. \textbf{30} (2017), no.~1,
  159--203.

\bibitem[SV87]{SV}
John Smillie and Karen Vogtmann, \emph{A generating function for the {E}uler
  characteristic of {${\rm Out}(F_n)$}}, Proceedings of the {N}orthwestern
  conference on cohomology of groups ({E}vanston, {I}ll., 1985), vol.~44, 1987,
  pp.~329--348.

\bibitem[Vog02]{Vogt}
Karen Vogtmann, \emph{Automorphisms of free groups and outer space}, Geom.
  Dedicata \textbf{94} (2002), 1--31.

\bibitem[Vog06]{Vogtmann-ICM}
\bysame, \emph{The cohomology of automorphism groups of free groups},
  International {C}ongress of {M}athematicians. {V}ol. {II}, Eur. Math. Soc.,
  Z\"{u}rich, 2006, pp.~1101--1117.

\bibitem[Web07]{Webb}
Peter Webb, \emph{An introduction to the representations and cohomology of
  categories}, Group representation theory, EPFL Press, Lausanne, 2007,
  pp.~149--173.

\bibitem[Zim96]{Zim}
Bruno Zimmermann, \emph{Finite groups of outer automorphisms of free groups},
  Glasgow Math. J. \textbf{38} (1996), no.~3, 275--282.

\end{thebibliography}
\bibliographystyle{amsalpha}

\end{document}